\documentclass[11pt]{article}
\usepackage{latexsym,amsmath,color,amsthm,amssymb,epsfig,graphicx,mathrsfs}
\usepackage{graphicx}
\usepackage{fullpage}
\def\qed{\hfill \ifhmode\unskip\nobreak\fi\quad\ifmmode\Box\else$\hfill\Box$\fi\\ }

\newtheorem{thrm}{Theorem}[section]
\newtheorem{lemm}[thrm]{Lemma}
\newtheorem{propo}[thrm]{Proposition}
\newtheorem{coro}[thrm]{Corollary}
\newtheorem{defi}[thrm]{Definition}

\newcommand{\thm}{\begin{thrm}}
\newcommand{\xthm}{\end{thrm}}
\newcommand{\lem}{\begin{lemm}}
\newcommand{\xlem}{\end{lemm}}
\newcommand{\prf}{\begin{proof}}
\newcommand{\xprf}{\end{proof}}
\newcommand{\prop}{\begin{propo}}
\newcommand{\xprop}{\end{propo}}
\newcommand{\cor}{\begin{coro}}
\newcommand{\xcor}{\end{coro}}
\newcommand{\defn}{\begin{defi}}
\newcommand{\xdefn}{\end{defi}}
\newcommand{\conj}{\begin{conjecture}}
\newcommand{\xconj}{\end{conjecture}}

\renewcommand{\phi}{\varphi}

\newcommand{\cH}{\mathcal{H}}

\newcommand{\cF}{\mathcal{F}}
\newtheorem{theorem}{Theorem}

\newtheorem{conjecture}{Conjecture}
\newtheorem{definition}{Definition}

\newcommand{\twr}{{{\rm twr}}}

\begin{document}

\title{\vspace{-0.5in} The Erd\H os-Hajnal hypergraph Ramsey problem}

\author{Dhruv Mubayi\thanks{Department of Mathematics, Statistics, and Computer Science, University of Illinois, Chicago, IL, 60607 USA.  Research partially supported by NSF grant DMS-1300138. Email: {\tt mubayi@uic.edu}} \and Andrew Suk\thanks{Department of Mathematics, Statistics, and Computer Science, University of Illinois, Chicago, IL, 60607 USA. Supported by NSF grant DMS-1500153. Email: {\tt suk@uic.edu}.}}

\maketitle

\vspace{-0.3in}

\begin{abstract}
Given integers $2\le t \le k+1  \le n$, let $g_k(t,n)$ be the minimum $N$ such that every red/blue coloring of the $k$-subsets of $\{1, \ldots, N\}$ yields either a $(k+1)$-set containing $t$ red $k$-subsets, or an $n$-set with all of its $k$-subsets  blue. Erd\H os and Hajnal proved in 1972  that for fixed $2\le t \le k$, there are positive constants $c_1$ and $c_2$ such that
$$ 2^{c_1 n} < g_k(t, n) < \twr_{t-1} (n^{c_2}),$$
where $\twr_{t-1}$ is a tower of 2's of height $t-2$. They conjectured that the tower growth rate in the upper bound is correct.  Despite decades of work on closely related and special cases of this problem by many researchers, there have been no improvements of the lower bound for $2<t<k$. Here we settle the Erd\H os-Hajnal conjecture in almost all cases in a strong form, by  determining the correct tower growth rate, and in half of the cases we also determine the correct power  of $n$ within the tower.  Specifically, we prove  that if $2<t<k-1$ and $k - t$ is even, then
$$g_k(t, n) = \twr_{t-1} ( n^{k-t+1 + o(1)} ).$$
Similar results are proved for $k - t$ odd.

\end{abstract}

\section{Introduction}

A $k$-uniform hypergraph $H$ ($k$-graph for short) with vertex set $V$ is a collection of $k$-element subsets of $V$.  We write $K^k_n$ for the complete $k$-uniform hypergraph on an $n$-element vertex set.  Given two families of $k$-graphs $\mathcal{F}$, $\mathcal{G}$, the \emph{Ramsey number} $r(\mathcal{F}, \mathcal{G})$ is the minimum $N$ such that every red/blue coloring of the edges of $K^k_N$ results in a monochromatic red copy of $F \in \mathcal{F}$ or a monochromatic blue copy of $G \in \mathcal{G}$.  In order to avoid the excessive use of superscripts, we use the simpler notation

$$r_k(\mathcal{F},n) = r(\mathcal{F},K^k_n) \hspace{1cm}\textnormal{and}\hspace{1cm} r_k(s,n) = r(K^k_s,K^k_n).$$

 Estimating the Ramsey number $r_k(s,n)$ is a fundamental problem in combinatorics and has been extensively studied since 1935.  For graphs, classical results of Erd\H os \cite{E47} and Erd\H os and Szekeres \cite{ES35} imply that $2^{n/2} < r_2(n,n) < 2^{2n}$.  For $k$-graphs with $k\geq 3$, Erd\H os, Hajnal, and Rado \cite{EHR,EH72} showed that there are positive constants $c_1$ and $c_2$ such that
$$\twr_{k-1}(c_1n^2) \leq r_k(n,n) \leq \twr_k(c_2n),$$

\noindent  where the \emph{tower function} is defined recursively as $\twr_1(x) = x$ and $\twr_{i + 1} = 2^{\twr_i(x)}.$  It is a major open problem to determine if $r_k(n,n) \geq \twr_k(cn)$ and Erd\H os offered a \$500 reward for a proof (see~\cite{Chung}).

In order to shed more light on these questions,  Erd\H os and Hajnal~\cite{EH72} in 1972 considered the  following more general parameter.

\begin{definition} For integers $2\le k < s <n$ and $2 \le t \le {s \choose k}$, let $r_k(s,t;n)$ be the minimum $N$ such that every red/blue coloring of  the edges of $K^k_N$ results in a monochromatic blue copy of $K_n^k$ or has a set of $s$ vertices which contains at least $t$ red edges.
\end{definition}
Note that by definition $r_k(s,n) = r_k(s, {s \choose k}; n)$ so $r_k(s,t;n)$ includes classical Ramsey numbers.


 The main conjecture of Erd\H os and Hajnal states that as $t$ grows from $1$ to ${s\choose k}$, there is a well-defined value $t_1=h_1^{(k)}(s)$ at which $r_k(s,t_1-1;n)$ is polynomial in $n$ while $r_k(s,t_1;n)$ is exponential in a power of $n$, another well-defined value $t_2=h_2^{(k)}(s)$ at which it changes from exponential to double exponential in a power of $n$ and so on,  and finally a well-defined value $t_{k-2}=h_{k-2}^{(k)}(s)<{s \choose k}$
  at which it changes from $\twr_{k-2}$ to $\twr_{k-1}$ in a power of $n$. They were not able to offer a conjecture as to what $h_i^{(k)}(s)$ is in general, except when $i=1$ and when $s=k+1$.

  $\bullet$ When $i=1$, they conjectured that $t_1=h_1^{(k)}(s)$ is one more than the number of edges in the $k$-graph obtained from a complete $k$-partite $k$-graph on $s$ vertices with almost equal part sizes, by repeating this construction  recursively within each part. Erd\H os offered \$500 for a proof of this (see \cite{Chung}).

 $\bullet$ When $s=k+1$, they conjectured that $h_i^{(k)}(k+1)=i+2$ and proved this for $i=1$ via the following:

\begin{theorem} [Erdos-Hajnal~\cite{EH72}] \label{eh}
For $k \ge 3$, there are  positive $c= c(k)$ and $c'=c'(k)$ such that
$$r_k(k+1, 2; n) < cn^{k-1} \qquad \hbox{ and } \qquad
r_k(k+1, 3; n) > 2^{c'n}.$$
\end{theorem}

  They also stated that the methods of Erd\H os and Rado \cite{ER} show that for $4 \le t \le k+1$ there exists $c=c(k,t)>0$ such that
\begin{equation} \label{ehupper} r_k(k+1, t; n) \leq \twr_{t-1}(n^c).
\end{equation}
 Hence in  our notation their conjecture for $s=k+1$ is the following lower bound:

\begin{conjecture} [Erdos-Hajnal~\cite{EH72}] \label{ehconj}
For $2 \le t \le k$, there exists  $c=c(k,t)>0$ such that
$$r_k(k+1, t; n) \geq \twr_{t-1}(c \, n).$$
\end{conjecture}

Another important special case of the Erd\H os-Hajnal problem is when $s$ is fixed and $t={s \choose k}$. Here we are interested in determining the correct tower growth rate for the \emph{off-diagonal} Ramsey number $r_k(s,n)$ as $n$ grows.  It follows from well-known results that $r_k(s,n) \leq \twr_{k-1}(n^{c})$ where $c = c(k,s)$ (see \cite{AKS,B,BK,ER} for the best known bounds) and the Erd\H os-Hajnal conjecture  implies  that for all $s\geq k+1$ there exists $c'>0$ such that
$$r_{k}(s,n) \ge r_k(s, h_{k-2}^{(k)}(s);n)> \twr_{k-1}(c'n).$$

The Erd\H os-Hajnal stepping up lemma shows this for all $s \geq 2^{k-1} - k + 3$ and Conlon, Fox, and Sudakov~\cite{CFS13} improved this to $s\geq \lceil 5k/2\rceil - 3$. 
 Recently the current authors  \cite{MS2}, and independently Conlon, Fox, and Sudakov \cite{CFS1}, further improved this to $s\geq k+3$.  For the remaining two values $s = k+2$ and $k+1$ the current authors  improved the previous best bounds to $\twr_{k-2}(n^{c\log  n})$ and $\twr_{k-2}(n^{c\log\log n})$ respectively \cite{MS1,MS2}.

As we have indicated, the function $r_k(s,t;n)$ encompasses several fundamental problems which have been studied for a while. In addition to off-diagonal and diagonal Ramsey numbers already mentioned, the case
$(k,s,t,n)=(k, k+1, k+1, k+1)$ has been studied in the context of the Erd\H os-Szekeres theorem and Ramsey numbers of ordered tight paths by several researchers~\cite{DLR, EM, FPSS, MS, MSW}, the more general case $(k,k+1, k+1, n)$ is related to high dimensional tournaments~\cite{LM}, and
even the very special case
$(3, 4, 3, n)$ has tight connections to quasirandom hypergraph constructions~\cite{BR, KNRS, LM1, LM2}. Lastly, $h_1^{(3)}(s)$ is quite well understood due to results of Erd\H os-Hajnal~\cite{EH72} and Conlon-Fox-Sudakov~\cite{CFS}.
In spite of this activity,  no lower bound better than
\begin{equation} \label{lower} r_k(k+1, t; n) \geq 2^{cn}\end{equation}
has been proved for a single pair $(k,t)$ with $2<t<k$.

In this paper we prove such lower bounds and settle Conjecture~\ref{ehconj} in almost all cases, while also improving the upper bounds. In half of the cases we even obtain the correct power of $n$ within the tower thereby improving Theorem~\ref{eh}, (\ref{ehupper}) and
(\ref{lower}).   First, we state our result when $t=3$.
\medskip

\begin{theorem}\label{t3}

For $k\geq 3$, there are positive  $c = c(k)$ and $c' = c'(k)$ such that
$$2^{c'n^{k-2}\log n } \geq r_k(k + 1,3;n) \geq 2^{cn^{k-2}}.$$
\end{theorem}

Our result for larger $t$, which represents the main new advance in this work,  is similar.
\medskip

\begin{theorem} \label{main}
For $4 \le t \le k-2$, there are positive $c =c(k,t)$ and $c'=c'(k,t)$ such that
$$\twr_{t-1}(c' n^{k-t + 1}\log n) \, \ge \, r_k(k+1,t; \, n)
 \, \ge \,  \begin{cases}
\twr_{t-1}(c \, n^{k-t + 1}) \qquad  \hbox{ if $k-t$ is even}\\
\twr_{t-1}(c \, n^{(k-t + 1)/2}) \hskip8pt \hbox{ if $k-t$ is odd.}
\end{cases}
$$
\end{theorem}

{\bf Remarks.}

$\bullet$ The lower bound in Theorem~\ref{t3} shows that the lower bound of Erd\H os and Hajnal (Theorem 9 in \cite{EH72}) is incorrect for $k\geq 4$.

$\bullet$ The basic approach for the proof of the tower lower bounds in Theorem~\ref{main} is the stepping up technique. Although  this method first appeared in 1965 (Lemma 6 of Erd\H os-Hajnal-Rado~\cite{EHR}), and has been used extensively since then by many researchers to solve various cases of the Erd\H os-Hajnal problem, it has not yielded any progress on Conjecture~\ref{ehconj}. It is the new ingredients that we add to the stepping up method that allow us to prove our lower bounds. These new ingredients have already been used successfully for the Erd\H os-Rogers problem~\cite{MS1} and we expect more applications in the future.

$\bullet$ The upper bounds in Theorem~\ref{main} also hold when $k-2<t\le k+1$.

$\bullet$ The proof of the lower bound in Theorem~\ref{main} together with the previous remark  yields
$$\twr_{k-3}(c \, n^3) \le r_k(k+1, k;\, n)  \le \twr_{k-1}(c' \, n)$$
and
$$\twr_{k-3}(c \, n^3) \le  r_k(k+1, k-1;\, n)  \le \twr_{k-2}(c' \, n^2).$$
Note that both these results are new.

\section{The upper bound}
In this section, we prove the upper bounds in Theorems \ref{t3} and \ref{main}.
The standard Erd\H os-Rado argument \cite{ER} gives an upper bound of the form
\begin{equation}\label{erupper}
r_k(k+1, t; n)< 2^{ {r_{k-1}(k, t-1; n-1) \choose k-1}}.\end{equation}
We improve this substantially in the case $t=3$ (to the optimal power of $n$) by adapting the on-line approach of Conlon, Fox, and Sudakov \cite{CFS}.

  Here we consider \emph{ordered} $(k-1)$-uniform hypergraphs $H = (V,E)$, that is, hypergraphs whose vertex set $V  = \{v_1,\ldots, v_n\}$ and $v_1 < \cdots < v_n$.  We define the ordering $\prec$ on ${V\choose k-2}$ such that for $a= (a_1,\ldots, a_{k-2})$ and $b = (b_1,\ldots, b_{k-2})$, where $a_1 < \cdots < a_{k-2}$ and $b_1< \ldots < b_{k-2}$, we have $a \prec b$ if the maximum $i$ for which $a_i \neq b_i$ we have $a_i  <b_i$.  This is the \emph{colex} order on sequences.

Consider the following game, played by two players, builder and painter:  at stage $i + 1$ a new vertex $v_{i + 1}$ is revealed; then, for every $(k-2)$-tuple $S$ among the existing vertices $v_1,\ldots, v_i$, builder decides, one by one, whether to draw the $(k-1)$-edge $S\cup \{v_{i +1}\}$; if he does expose such an edge, painter has to color it either red or blue immediately.  The exposed vertices will be naturally ordered $v_1 < v_2< \cdots < v_i$.

 Let $F$ be the ordered $(k-1)$-graph on $k$ vertices $a_1,\ldots, a_k$, where $a_1 < \cdots < a_k$, with two edges $(a_1,\ldots, a_{k-1})$ and $(a_1,\ldots, a_{k-2},a_k)$.


   Our main result in this section is the following.

\begin{theorem}\label{builder}
In the vertex on-line ordered Ramsey game, builder has a strategy which ensures a red $F$ or a blue $K^{k-1}_n$ using at most $s = O(n^{k-2})$ vertices, $r = O(n^{k-2})$ red edges, and $m = O(n^{2k-4})$ total edges.

\end{theorem}

For $t \ge 2$, let $\cF_t^k$ be the collection of ordered $k$-graphs with vertex set $a_1< \cdots <a_{k+1}$ and $t$ edges, where one of the edges is $(a_1, \ldots, a_k)$. Define $r_k(\cF_t^k,n)$ to be the minimum $N$ such that in every red/blue coloring of ${[N] \choose k}$ there is a red ordered $H \in \cF_t^k$ or a blue $K_n^k$.  We need the following result  which is a straightforward adaptation of Theorem 2.1 in \cite{CFS}.

\begin{theorem}
Suppose in the vertex on-line ordered Ramsey game that builder has a strategy which ensures a red $F$ or a blue $K^{k-1}_n$ using at most $s$ vertices, $r$ red edges, and in total $m$ edges.  Then, for any $0 < \alpha < 1/2$,
$$r_k(\cF_3^k,n) \leq O(s\alpha^{-r}(1 - \alpha)^{r-m}).$$
\end{theorem}

Setting $\alpha = n^{-2k +4}$ and using Theorem~\ref{builder} we obtain
\begin{equation} \label{final}
r_{k}(\cF_3^k, n)< 2^{O(n^{k-2}\log n)}.\end{equation}

Now we apply (\ref{erupper}) $t-3$ times to obtain
$$r_k(k+1, t; n) < \twr_{t-2} (c \, r_{k-t+3}(k-t+3, 3; n)) \le
\twr_{t-2} (c \, r_{k-t+3}(\cF_3^{k-t+3}, n)).$$
Note that this application of (\ref{erupper}) involves some subtleties, in particular, when applying the standard Erd\H os-Rado pigeonhole argument \cite{ER}, one should observe that a copy of a member of $\cF_{t-1}^{k-1}$ gives rise to a copy of some member of $\cF_{t}^{k}$ due to the existence of the initial edge, and this edge remains to carry out the induction.

Finally we apply (\ref{final}) with $k$ replaced by $k-t+3$ to obtain
$$r_k(k+1, t; n) < \twr_{t-1}(c' n^{k-t+1} \log n)$$
as desired.  We now turn to the proof of Theorem~\ref{builder}.
\bigskip

\emph{Proof of Theorem \ref{builder}.}  The proof is based on the following strategy for builder.  During the game, we will label each exposed vertex with a string of $R$'s and $B$'s, and build a subset $T$ of the exposed vertices as follows.  We start the game by exposing $k-2$ vertices $v_1,\ldots, v_{k-2}$, $v_1 < \cdots < v_{k-2}$, and label each of these vertices by $\emptyset$.  We set $T = \{v_1,\ldots, v_{k-2}\}$.  For every other vertex exposed during this game, builder will draw an edge between that vertex and $\{v_1,\ldots, v_{k - 2}\}$.  Recall painter will immediately color that edge red or blue.  The first exposed vertex that is connected to $\{v_1,\ldots, v_{k-2}\}$ with a red (blue) $(k-1)$-edge will be labelled $R$ ($B$).  The first exposed vertex that is connected to $\{v_1,\ldots, v_{k -2}\}$ with a blue edge, will be added to the set $T$.

Now assume that the vertices $v_1,\ldots, v_i$ have been exposed, each such vertex is labeled with a string of $R$'s and $B$'s, and some $(k-1)$-edges between these vertices have been drawn (and colored).  Moreover, we have our current subset $T\subset \{v_1,\ldots, v_i\}$. The next stage begins by exposing vertex $v_{i + 1}$, and builder will always begin by drawing the $(k-1)$-edge $(v_1,\ldots, v_{k - 2}, v_{i + 1})$.  Depending on whether painter colors this edge red or blue, the first digit in the string assigned to $v_{i + 1}$ will be $R$ or $B$ respectively.  Builder will continue to draw edges of the form $S\cup \{v_{i + 1}\}$, where $S\subset T$ and $|S| = k-2$, one by one by considering each such $S$ in order with respect to~$\prec$.  If at any point painter colors any drawn edge red, we immediately stop and proceed to the next stage, revealing the new vertex $v_{i + 2}$.  If painter paints every such edge blue, then we add $v_{i + 1}$ to $T$, and proceed to the next stage. Note that all $(k-1)$-subsets of $T$ are blue.

Let us make some observations about this game.  Suppose at the end of a stage, builder has not forced a red $F$ or a blue $K^{k-1}_n$.  Then:

1) a vertex is labelled with a string of $B$'s (no $R$'s) if and only if it is in $T$,

2) a vertex is labelled with a sequence of $B$'s followed by a single $R$ if and only if it is not in $T$,

3) no two strings have $R$ in the same position,

4) no two string are the same,

5) no string has more than ${n\choose k-2}$ $B$'s.

1) holds since this is how $T$ is created. 2) holds since we stop the moment we get an $R$.  3) holds since if we obtained two such strings, and since builder's strategy was to consider $(k-2)$-tuples $S\subset T$ in colex order, we would obtain two red $(k-1)$-tuples sharing the same first $k-2$ points $S \subset T$, and this yields a red $F$. Note that we are using the crucial fact that the colex order $\prec$ takes all $(k-2)$-subsets of a set before moving on to a new element.
4) holds since the previous three properties show that this is possible only if two vertices in $T$ have the same label (number of $B$'s).  However this is impossible since the larger vertex considers more edges and must have more $B$'s.  5) holds since otherwise $T$ would induce a blue $K^{k-1}_n$ due to the property of $\prec$ mentioned above.

Therefore, in the vertex on-line Ramsey game, builder has a strategy which ensures a red $F$ or a blue $K^{k-1}_n$ using at most $O(n^{k-2})$ vertices, $O(n^{2k-4})$ edges, and $O(n^{k-2})$ red edges. \qed

\section{The lower bound}

In this section we prove the lower bounds in Theorems \ref{t3} and \ref{main}. We start with Theorem \ref{t3}, the case when $t=3$, which has no dependence on parity. We point out that it is this result which shows that the lower bound of Erd\H os and Hajnal in Theorem 9 of \cite{EH72} is incorrect for $k\geq 4$.  Note that their lower bound corresponds to the upper bound $2^{c_k n \log n}$ in our notation.

\begin{theorem} \label{ehimprove}
For $k \ge 3$ there exists $c_k>0$ such that
$$r_k(k+1, 3; n) >2^{c_k n^{k-2}}.$$
\end{theorem}

\proof Let $k\geq 3$ and $N=2^{c_kn^{k-2}}$ where $c_k$ is a sufficiently small constant that will be determined later. Color the $(k-1)$-sets of $[N]$ randomly with $k$ colors, where each edge has probability $1/k$ of being a particular color independent of all other edges. Call this coloring $\phi$ and suppose that $\phi(S) \in [k]$ for all $S \in {[N]\choose k-1}$. Now, given a $k$-set $e=a_1<\cdots <a_{k}$ of $[N]$, and a $(k-1)$-subset $S=e-\{a_i\}$ of $e$, let rank$_e(S)=i$. Define the red/blue coloring $\chi$ of ${[N] \choose k}$ by
$$\chi(e)=red \qquad \hbox{ iff } \qquad \phi(S)=\hbox{rank}_e(S) \, \hbox{ for all $S \in {e \choose k-1}$}.$$  The probability that $\chi(e)=red$ is $p_k=k^{-k}$. If an $n$-set $X$ is blue, then all the $k$-subsets of $X$ in a partial Steiner system $F=Sp(k-1, k, n)$  of $X$ are blue, and the colors within $F$ are assigned independently as they depend only on $(k-1)$-subsets
(recall that by definition of $F$, $|A \cap B|\le k-2$ for all $A, B \in F$). It is well known that there exists an $F$ so that $|F|=\Theta(n^{k-1})$ as $n$ grows.  Hence the expected number of blue copies of $K_n^k$ in $[N]$ is at most
$${N \choose n} (1-p_k)^{|F|} < N^n e^{-\Theta(n^{k-1})} < (Ne^{-\Theta(n^{k-2})})^n <1$$
due to the choice of $c_k$. So there exists a coloring $\chi$ with no blue $K_n^k$. Next, suppose for contradiction that $Y$ is a $(k+1)$-subset of $[N]$ that contains three red edges under $\chi$. Say that $Y=a_1<\cdots <a_{k+1}$. Call these red edges $e_i, e_j, e_l$, where $e_q=Y-\{a_q\}$.  Assume that $i<j<l$ so that $l-1>i$ and let $S=Y-\{a_i, a_l\}$. Then
$$\phi(S)=\hbox{rank}_{e_i}(S)=l-1 >i= \hbox{rank}_{e_l}(S)=\phi(S).$$
This contradiction shows that we have at most two red edges in $Y$.   \qed

For larger values of $t$, we establish the lower bound in Theorem~\ref{main} using ideas that originated in the Erd\H os-Hajnal stepping up lemma (see \cite{GRS}) and were further developed recently by the authors in~\cite{MS1}.
As mentioned in the introduction, this is the main new advance in this work.

 It is convenient to use the following notation. Let $\cH_{t}:= \cH^k_{t}$  be the family of $k$-graphs with $k+1$ vertices and $t$ edges, and define
$r_k(\cH_t^k, n)=r_k(k+1, t;n)$.  We will omit the superscript if it is obvious from the context.  In what follows, by a red $\cH_t$ we mean a red copy of some member $H \in \cH_t$.
\bigskip

\begin{theorem}\label{stepupeasy}
Let $k\geq 6$ and $t \geq 4$.  If we are not in the case when $t = 4$ and $k$ is odd, then we have $r_k(\cH_t, 2kn) > 2^{r_{k-1}(\cH_{t-1}, n)-1}.$

\end{theorem}

\begin{proof} Set $N=r_{k-1}(\cH_{t-1}, n) -1$, and let $\phi$ be a red/blue coloring of the edges of $K_N^{k-1}$ with no red $\cH_{t-1}^{k-1}$ and no blue $K_{n}^{k-1}$.  Given $\phi$, we will produce a red/blue coloring $\chi$ on the edges of $K_{2^N}^k$ with no red $\cH_t^k$ and no blue $K_{2kn}^k$. Let $V(K_N^{k-1})= [N]$ and $V(K_{2^N}^k)= \{0,1\}^{N}$.

The vertices of $V(K_{2^N}^k)$ are naturally ordered by the integer they represent in binary, so for $a,b \in V(K_{2^N}^k)$ where $a = (a(1),\ldots, a(N))$ and $b = (b(1),\ldots, b(N))$, $a < b$ iff  there is an $i$ such that $a(i)  = 0$, $b(i) = 1$, and $a(j)=b(j)$ for all $1 \le j <i$. In other words,  $i$ is the first position (minimum index) in which $a$ and $b$ differ.  For $a,b\in V(K_{2^N}^k)$, where $a\neq b$, let $\delta(a,b)$ denote the least $i$ for which $a(i)\neq b(i)$.\footnote{For $a = (1,0,1,1,0)$ and $b = (1,0,0,1,1)$, we have $a > b$ and $\delta(a,b) = 3$.  Let us remark that we have slightly modified the definition of $\delta$ in \cite{GRS}, using \emph{least} $i$ rather than \emph{largest} $i$ for which $a(i)\neq b(i)$.}  Notice we have the following stepping-up properties (see \cite{GRS}).

\begin{description}

\item[Property A:] For every triple $a < b < c$, $\delta(a,b) \not = \delta(b,c)$.

\item[Property B:] For $a_1 < \cdots < a_r$, $\delta(a_1,a_{r}) = \min_{1 \leq j \leq r-1}\delta(a_j,a_{j + 1})$.

\end{description}

Before we define the coloring $\chi$, let us introduce several definitions.  Set $V =\{0,1\}^N$.  Given any $m$-set $S = \{a_1,\ldots, a_m\}\subset V$, where $a_1<a_2<\cdots<a_{m}$, consider the integers $\delta_i=\delta(a_i,a_{i+1}), 1\le i\le m-1$. We say that $\delta_i$ is a {\it local minimum} if $\delta_{i-1}>\delta_i<\delta_{i+1}$, a {\it local maximum} if $\delta_{i-1}<\delta_i>\delta_{i+1}$, and a {\it local extremum} if it is either a local minimum or a local maximum.  Since $\delta_{i-1} \not = \delta_i$ for every $i$, every nonmonotone sequence $\{\delta_i\}_{i = 1}^{m-1}$ has a local extremum. For convenience, we write $\delta(S) = \{\delta_i\}_{i = 1}^{m}$.

We now define the coloring $\chi$ on the $k$-tuples of $V=\{0,1\}^N$ as follows.  Given an edge $e= (a_1,\ldots, a_k)$ in $V=V(K_{2^N}^k)$, where $a_1< \cdots < a_k$, let $\delta_i = \delta(a_i,a_{i+1})$.  Then $\chi(e)=red$ if

$\bullet$ the sequence $\delta(e)$ is monotone and $\phi(\delta_1, \ldots, \delta_{k-1}) =red$, or

$\bullet$ the sequence $\delta(e)$ is a \emph{zigzag}, meaning $\delta_2,\delta_4,\ldots$ are local minimums and $\delta_3,\delta_5,\ldots$ are local maximums. In other words,  $\delta_1 > \delta_2 < \delta_3 > \delta_4 < \cdots$.

Otherwise $\chi(e)=blue$.

 Note that the definition of zigzag  requires $\delta_1 > \delta_2 < \delta_3 > \delta_4 < \cdots$ and if the inequalities are in the opposite directions, i.e.  $\delta_1 <\delta_2 >\delta_3 < \cdots$, then $e$ is not zigzag.

The following property can easily be verified using Properties A and B (see \cite{GRS}).

\begin{description}

\item[Property C:] For $a_1 < \cdots < a_r$, set $\delta_j = \delta(a_j,a_{j + 1})$ and suppose that $\delta_1,\ldots, \delta_{r-1}$ forms a monotone sequence.  If $\chi$ colors every $k$-tuple in $\{a_1,\ldots, a_r\}$ red (blue), then $\phi$ colors every $(k-1)$-tuple in $\{\delta_1,\ldots, \delta_{r-1}\}$ red (blue).

\end{description}

Set $m = 2kn$.  For sake of contradiction, suppose there is an $m$-set $S = \{a_1,\ldots, a_m\} \subset V$ such that $\chi$ colors every $k$-tuple in $S$ blue.  Let $\delta_i = \delta(a_i,a_{i+1})$ for $1\leq i \leq m-1$.  By Property C, there is no integer $j$ such that the sequence $\{\delta_i\}_{i = j}^{j + n - 1}$ is monotone, since otherwise $\phi$ colors every triple of the $n$-set $\{\delta_j,\delta_{j+1},\ldots, \delta_{j + n-1}\}$ blue which is a contradiction.  Therefore, we can assume there are $k$ consecutive extremums $\delta_{i_1},\delta_{i_2},\cdots, \delta_{i_{k}}$, such that $\delta_{i_1},\delta_{i_3},\ldots$ are local maximums and  $\delta_{i_2},\delta_{i_4},\ldots$ are local minimums.  Recall that $\delta_{i_j} = \delta(a_{i_j},a_{i_j + 1})$. For $k$ even, consider the $k$ vertices
$$e = (a_{i_1}, a_{i_1 + 1},a_{i_3},a_{i_3 + 1}, a_{i_5}, a_{i_5 + 1}, \ldots , a_{i_{k-1}}, a_{i_{k-1}  +1}).$$
By Property B, we have
$$\delta(a_{i_1},a_{i_1 + 1}) > \delta(a_{i_1 + 1},a_{i_3}) < \delta(a_{i_3},a_{i_3 + 1}) > \delta(a_{i_3+ 1}, a_{i_5}) < \cdots.$$
Hence $\delta(e)$ is zigzag and $\chi(e) = red$, contradiction.   For $k$ odd, consider the $k$ vertices
$$e = (a_{i_1}, a_{i_1 + 1},a_{i_3},a_{i_3 + 1}, a_{i_5}, a_{i_5 + 1}, \ldots , a_{i_{k-2}}, a_{i_{k-2}  +1}, a_{i_{k-2} + 2}).$$
Again by Property B, we have
$$\delta(a_{i_1},a_{i_1 + 1}) > \delta(a_{i_1 + 1},a_{i_3}) < \delta(a_{i_3},a_{i_3 + 1}) > \delta(a_{i_3+ 1}, a_{i_5}) < \cdots,$$
which implies $\delta(e)$ is zigzag and $\chi(e) = red$, contradiction.

Now it suffices to show that there is no red copy of an $H \in \cH_{t}^k$ under $\chi$.  We first establish the following claim.  We use the notation $X+x=X \cup \{x\}$.

{\bf Claim 1.} For $k\geq 6$, let $e,e' \in E(K^k_{2^N})$ such that $\delta(e)$ is monotone, $\delta(e')$ is zigzag, and $\chi(e)=\chi(e')=red$.  Then we have $|e \cap e'|<k-1$.

\proof  Suppose for contradiction that $e= (a_1,\ldots, a_k)$, where $a_1< \cdots < a_k$, and $e'=e-a_i+a'$.
Now the sequence $\delta(e-a_i)$ is monotone of length $k-1$.  Relabel $e-a_i = (b_1,\ldots, b_{k-1})$, where $b_1<\cdots < b_{k-1}$ and let us insert $a'$ into $e-a_i$.  Assume first that sequence $\delta(e-a_i)$ is increasing. If $a'>b_3$, then $\delta(b_1,b_2)< \delta(b_2,b_3)$ shows that $\delta(e')$ is not zigzag (starts in the wrong direction). If $a'<b_2$, then $\delta(b_2,b_3)< \delta(b_3,b_4) < \delta(b_4,b_5)$ shows that $\delta(e')$ is not zigzag.  So
we have $b_2< a'< b_3$.  However by Property B, $\delta(b_2,a') \geq \delta(b_2,b_3)$. So we have $$\delta(b_2,a') \geq \delta(b_2,b_3) > \delta(b_1,b_2),$$ which shows that $\delta(e')$ is not zigzag

Now suppose $\delta(e-a_i)$ is decreasing.  For $\delta(e')$ to be zigzag, we must have $b_2 < a' < b_4$.  If $b_2 < a' < b_3$, then by Property B we have $\delta(a',b_3)> \delta(b_3,b_4) > \delta(b_4,b_5)$ which is a contradiction.  Now if $b_3 < a' < b_4$, we must have $\delta(b_2,b_3) < \delta(b_3,a')>\delta(a', b_4)<\delta(b_4, b_5)$.  However, Property B and the fact that
$\delta(e-a_i)$ is decreasing imply that
 $\delta(a',b_4) \ge \delta(b_3,b_4) > \delta(b_4,b_5)$ which is a contradiction. \qed

For sake of contradiction, suppose $\chi$ produces a red $H \in \cH_t$.  By Claim 1, we may assume that for the $t$ red edges $e_1,\ldots, e_t \in E(H)$, either all of the sequences $\delta(e_1),\ldots, \delta(e_t)$ are monotone or all of them are zigzag.   Let $V(H) = a = \{a_1,\ldots, a_{k + 1}\}$ with $a_1<\cdots < a_{k+1}$ and $\delta_i = \delta(a_i,a_{i + 1})$.

{\bf Case 1.} All $t$ sequences are monotone.
Suppose they are all increasing (clearly one cannot be increasing and another decreasing).  Then one can easily see that $\delta(a)$ is increasing.  By Property B, for $i \leq k$, we have $\delta(a - a_i) = \delta(a) - \delta_i$ and $\delta(a-a_k) = \delta(a- a_{k + 1})$.  Hence these $t$ red edges give rise to at least $t-1$ red edges in $K_N^{k-1}$, which is a contradiction.  If all $t$ sequences are decreasing, then a similar argument follows.

 \textbf{Case 2.}  All $t$ sequences are zigzag, and $t\geq 5$.  Since $t \ge 5$, we must have two red edges $e_1 = a-a_i$ and $e_2 = a - a_j$ where $|i-j| \geq 4$.  Several times we will use the following

 {\bf Fact:} $\delta_i \ne \delta_{i+2}$ as long as $\delta_{i+1}>\delta_i$.

\emph{Case 2.1}.  If $i=1$ then we have $j\ge 5$.  Since $\delta(e_2)$ is zigzag, we have
$\delta_1>\delta_2<\delta_3$, but this contradicts the fact that $\delta(e_1)$ is zigzag as $\delta_3$ is a local maximum in the sequence.

\emph{Case 2.2}.  If $i = 2$, then we have $j \geq 6$.  Since $\delta(e_2)$ is zigzag, we have
$\delta_1>\delta_2<\delta_3>\delta_4$.  By Property B, $\delta(e_1)$ is not zigzag since $\delta_3$ is a local maximum in the sequence.

\emph{Case 2.3}. If $i\geq 3$, then we can also conclude that $i \leq k-3$ since $t\geq 5$.  Suppose $\delta_i$ is a local minimum in the sequence $\delta(e_2)$.  Since $\delta(e_2)$ is zigzag, we have $\delta_{i-2} <  \delta_{i-1} > \delta_{i } < \delta_{i + 1}$.  Moreover, we can conclude that $i-2 \geq 2$ since otherwise $\delta(e_2)$ is not zigzag (wrong direction). Hence $\delta_{i-3} > \delta_{i-2} <  \delta_{i-1} > \delta_{i } < \delta_{i + 1}$. By Property B, we have $\delta(a_{i-1},a_{i + 1}) = \delta_i$. By the Fact, $\delta_i \ne \delta_{i-2}$.  If $\delta_i  < \delta_{i-2}$, then $\delta(e_1)$ is not zigzag as $\delta_{i-3} > \delta_{i-2} > \delta(a_{i-1},a_{i + 1})$.   If $\delta_i > \delta_{i -2}$, then again $\delta(e_1)$ is not zigzag as $\delta_{i - 2} < \delta(a_{i - 1},a_{i+1}) < \delta_{i + 1}$, contradiction.

Now suppose $\delta_i$ is a local maximum in the sequence $\delta(e_2)$.  Then we have $\delta_{i-2} > \delta_{i-1} < \delta_i > \delta_{i + 1} < \delta_{i + 2}$. By Property B, we have $\delta(a_{i - 1},a_{i + 1}) = \delta_{i - 1}$.  By the Fact, $\delta_{i-1} \ne \delta_{i+1}$. If $\delta_{i -1} > \delta_{i + 1}$, then $\delta(e_1)$ is not zigzag as $\delta_{i-2} > \delta(a_{i-1},a_{i + 1}) > \delta_{i + 1}$.  If $\delta_{i -1} < \delta_{i + 1}$, then again $\delta(e_1)$ is not zigzag as $\delta(a_{i-1},a_{i + 1}) < \delta_{i + 1} < \delta_{i + 2}$, contradiction.

{\bf Case 3.} Suppose $t = 4$, $k$ is even, and all four sequences are zigzag.  Let $e_j = a - a_{i_j}$ be the four red edges on the vertex set $a = \{a_1,\ldots, a_{k + 1}\}$, such that $i_1 < i_2 < i_3 < i_4$.    We copy the argument in Case 2 verbatim unless $i_4 - i_1 = 3$, and therefore we can assume the four red edges are of the form $e_1 = a - a_i, e_2 = a - a_{i + 1}, e_3 = a - a_{i + 2}, e_4 = a - a_{i + 3}$.

\emph{Case 3.1}.  Suppose $i =1$.  Since $\delta(e_4)$ is zigzag, we have $\delta_1 > \delta_2 < \delta(a_3,a_5)$.  By Property B, $\delta_3 \ge \delta(a_3, a_5) > \delta_2$, which implies $\delta(e_1)$ cannot be zigzag (wrong direction).

\emph{Case 3.2.}  Suppose $i = 2$. Since $\delta(e_4)$ is zigzag, we have $\delta_1 > \delta_2 < \delta_3$.  By Property B, we have $ \delta(a_1,a_3)=\delta_2  < \delta_3$ which implies $\delta(e_1)$ is not zigzag (wrong direction), contradiction.

\emph{Case 3.3}.  Suppose $i = k-2$.  This is the only part of the proof that requires $k$ to be even. Since $k$ is even and $\delta(e_1)$ is zigzag, we have $\delta_{k-4} < \delta(a_{k-3},a_{k-1}) > \delta_{k-1} < \delta_k$.  By Property~B,
$\delta_{k-2} \ge \delta(a_{k-3}, a_{k-1})> \delta_{k-1}$, and since $k$ is even, this implies $\delta(e_4)$ is not zigzag (wrong direction).

\emph{Case 3.4}.  Suppose $3 \leq i \leq k - 3$.  Then $\delta_i$ is an extremum in the sequence $\delta(e_4)$.  Suppose it is a local minimum, which implies $i \geq 4$.  Then we have $\delta_{i-3}> \delta_{i -2}< \delta_{i - 1} > \delta_i < \delta_{i + 1}$.  By Property~B we have $\delta(a_{i-1},a_{i + 1}) = \delta_i$.  By the Fact, $\delta_i \ne \delta_{i-2}$. If $\delta_i > \delta_{i-2}$, then we have $\delta_{i-2} < \delta(a_{i - 1},a_{i + 1}) < \delta_{i + 1}$, and hence $\delta(e_1)$ is not zigzag.  If $\delta_i < \delta_{i-2}$, then we have $\delta_{i-3} > \delta_{i-2} >  \delta(a_{i - 1},a_{i + 1})$, and hence $\delta(e_1)$ is not zigzag, contradiction.

Now suppose that $\delta_i$ is a local maximum in the sequence $\delta(e_4)$.  Then we have $\delta_{i-2} > \delta_{i - 1} < \delta_i > \delta_{i +1} < \delta(a_{i + 2},a_{i + 4})$.  By Property B, we have $\delta_{i + 2} \geq \delta(a_{i + 2},a_{i + 4})$ and $\delta(a_{i-1},a_{i + 1}) = \delta_{i -1}$.   By the Fact, $\delta_{i-1} \ne \delta_{i+1}$. If $\delta_{i-1} < \delta_{i + 1}$, then we have $\delta(a_{i - 1},a_{i + 1}) < \delta_{i + 1} < \delta_{i + 2}$, which implies $\delta(e_1)$ is not zigzag.  If $\delta_{i-1} > \delta_{i + 1}$, then we have $\delta_{i-2} > \delta(a_{i -1},a_{i +1}) > \delta_{i + 1}$, which implies $\delta(e_1)$ is not zigzag, contradiction.    \end{proof}

We now establish the stepping-up lemma for the special case $k$ is odd and $t = 4$.

\bigskip

\begin{theorem}\label{3to4odd}For odd $k >6$, we have $r_k(\cH_4, 4n^2) > 2^{r_{k-1}(\cH_{3}, n)-1}.$

\end{theorem}

\begin{proof}  The proof is nearly identical to the previous proof  though there is one crucial difference in the definition of a red edge.  Again we set $N=r_{k-1}(\cH_{3}, n) -1$, and let $\phi$ be a red/blue coloring of the edges of $K_N^{k-1}$ with no red $\cH_{t-1}^{k-1}$ and no blue $K_{n}^{k-1}$.  Given $\phi$, we will produce a red/blue coloring $\chi$ on the edges of $K_{2^N}^k$ with no red $\cH_t^k$ and no blue $K_{4n^2}^k$. Let $V(K_N^{k-1})=[N]$ and $V(K_{2^N}^k)= \{0,1\}^{N}$, and order the elements of $V(K_{2^N}^k)$ by the integer they represent in binary.  We define $\chi$ slightly different from above.  Given an edge $e= (a_1,\ldots, a_k)$ in $K_{2^N}^k$, where $a_1< \cdots < a_k$, let $\delta_i = \delta(a_i,a_{i+1})$.  Then $\chi(e)=red$ if

$\bullet$ the sequence $\delta(e)$ is monotone and $\phi(\delta_1, \ldots, \delta_{k-1}) =red$, or

$\bullet$ the sequence $\delta(e)$ is a \emph{strong-zigzag}, meaning $\delta_2,\delta_4,\ldots$ are local minimums and $\delta_3,\delta_5,\ldots$ are local maximums, and $\delta_{k-1} < \delta_{k-3}$. In other words,
$$\delta_1 > \delta_2 < \delta_3 > \delta_4 < \cdots \cdots > \delta_{k-3} <\delta_{k-2} > \delta_{k-1} \qquad \hbox{ and } \qquad
 \delta_{k-1}<\delta_{k-3}.$$

Otherwise $\chi(e)=blue$.

Set $m = 4n^2$.  For sake of contradiction, suppose there is an $m$-set $S = \{a_1,\ldots, a_m\}$, where $a_1 < \cdots < a_m$, such that $\chi$ colors every $k$-tuple in $S$ blue.  Let $\delta_i = \delta(a_i,a_{i+1})$ for $1\leq i \leq m-1$ and consider the sequence $\delta(S)$.  By Property C, there is no integer $j$ such that the sequence $\{\delta_i\}_{i = j}^{j  + n -1}$ is monotone, since otherwise $\phi$ colors every $(k-1)$-tuple of the $n$-set $\{\delta_j,\delta_{j+1},\ldots, \delta_{j + n - 1}\}$ blue which is a contradiction.  Therefore, we can assume there are $2n$ consecutive extremums $\delta_{i_1},\delta_{i_2},\cdots, \delta_{i_{2n}}$, such that $\delta_{i_1},\delta_{i_3},\ldots,\delta_{i_{2n-1}}$ are local maximums and  $\delta_{i_2},\delta_{i_4},\ldots,\delta_{i_{2n}}$ are local minimums.  Recall that $\delta_{i_j} = \delta(a_{i_j},a_{i_j + 1})$.  Notice that by Properties A and B (or the Fact), we have $\delta_{i_{2n}} \neq \delta_{i_{2n-2}}$.  Suppose $\delta_{i_{2n}} < \delta_{i_{2n-2}}$, and consider the vertices corresponding the last $(k-1)/2$ local maximums along with $a_{i_{2n} +1}$.  More precisely, the vertices
$$a_{i_{2n - (k - 2)}}, a_{i_{2n - (k - 2)} + 1} , \ldots , a_{i_{2n -3}},  a_{i_{2n - 3} + 1}, a_{i_{2n - 1}},  a_{i_{2n - 1} + 1}, a_{i_{2n}+1}.$$
By the same argument as above, these vertices correspond to a strong-zigzag sequence, and therefore $\chi$ colors these $k$ vertices red and we have a contradiction.  Therefore we can assume $\delta_{i_{2n}} > \delta_{i_{2n-2}}$.  By the same argument, we can conclude that $\delta_{i_{2n - 2}} > \delta_{i_{2n - 4}}$.  After repeating this argument $n$ times, we have
$$\delta_{i_{2n}} > \delta_{i_{2n-2}} > \cdots > \delta_{i_2}.$$
Set $T = \{a_{i_2}, a_{i_4}, \ldots, a_{2n-2}, a_{2n}, a_{2n +1}\}$.  By Property B, $\delta(T)$ is a monotone sequence of length $n$, which implies $\phi$ created a blue clique of size $n$ in $K^{k-1}_N$, contradiction.

Now it suffices to show that there is no red copy of an $H \in \cH_{4}^k$ under $\chi$.  For sake of contradiction, suppose $\chi$ produces a red $H \in \cH_4$.   Let $V(H) = a = \{a_1,\ldots, a_{k + 1}\}$, $a_1 < \cdots < a_{k + 1}$, and $\delta_i = \delta(a_i,a_{i + 1})$.

We follow the same arguments as in Theorem~\ref{stepupeasy}, except we need to replace Case 3.3, since that was the only place where we used the fact that $k$ is even.  This is the case when our four red edges have the form $e_1 = a - a_{k-2}, e_2 = a - a_{k-1}, e_3 = a- a_k, e_4 = a - a_{k + 1}$.  Since $\delta(e_1)$ is strong-zigzag and $k$ is odd, we have $\delta(a_{k-3},a_{k -1}) < \delta_{k-1} > \delta_k$ and $\delta(a_{k-3},a_{k -1}) >  \delta_k$.  For $\delta(e_4)$ to be a strong-zigzag, we must have $\delta_{k-3} < \delta_{k-2} > \delta_{k - 1}$.  By Property B, this implies $\delta(a_{k-3},a_{k - 1}) = \delta_{k-3}$, and therefore $\delta_{k -3} < \delta_{k-1}$, which implies $\delta(e_4)$ is not a strong-zigzag, contradiction. \end{proof}

\emph{Proof of lower bound in Theorem \ref{main}}.  Suppose $k-t + 4$ is even.  Then by a $(t-3)$-fold application of Theorem~\ref{stepupeasy}, along with Theorem \ref{ehimprove}, we have $r_k(\cH_t,n) \geq \twr_{t-1}(c_1 n^{k-t +1})$ where $c_1 = c_1(k,t)$.  If $k- t + 4$ is odd, then by a $(t-4)$-fold application of Theorem \ref{stepupeasy}, along with Theorem \ref{3to4odd} and Theorem \ref{ehimprove}, we have $r_k(\cH_t,n) \geq \twr_{t-1}(c_2 n^{(k-t + 1)/2})$, where $c_2 = c_2(k,t)$.\qed

\end{document}